\pdfoutput=1 % FOR ARXIV
\documentclass[a4paper]{article}
\usepackage{amssymb,amsmath,amscd,amsfonts,amsthm,color}
\usepackage{makerobust}
 \MakeRobust{\overrightarrow}% FOR ARXIV, INSTEAD OF PREVIOUS LINE
\usepackage{graphicx}

\newcommand{\bb}[1]{\mathbb{#1}}
\newcommand{\cc}[1]{\mathcal{#1}}
\newcommand{\ra}[1]{\overrightarrow{#1}}
\newcommand{\la}[1]{\overleftarrow{#1}}
\def\specrel#1#2{\mathrel{\mathop{\kern0pt #1}\limits_{\!\!#2\!\!}}}

\newcommand{\bw}{{\rm bw}}% for block-width; same initials as for `branch-width" - do we care?
\newcommand{\pw}{{\rm pw}}% for profile-width; same initials as for `path-width" - do we care?
\newcommand{\tw}{{\rm tw}}% tree-width
\newcommand{\brw}{{\rm brw}}% `adjusted branch-width'

\newcommand{\B}{\mathcal B}
\newcommand{\F}{\mathcal F}
\renewcommand{\P}{\mathcal P}
\renewcommand{\S}{\mathcal S}
\newcommand{\T}{\mathcal T}
\newcommand{\V}{\mathcal V}

\newcommand{\vr}{\ra{r}}
\newcommand{\vs}{\ra{s}}
\newcommand{\rv}{\la{r}}
\newcommand{\sv}{\la{s}}
\newcommand{\vS}{\ra{S}}

\newcommand{\td}{tree-decom\-po\-si\-tion}

\newcommand\COMMENT[1]{}

\numberwithin{equation}{section}
\newtheorem{theorem}{Theorem}

\newtheorem{lemma}[theorem]{Lemma}
\newtheorem{corollary}[theorem]{Corollary}

\begin{document}
\title{Duality theorems for blocks and tangles in graphs}
\author{Reinhard Diestel, Philipp Eberenz and Joshua Erde}
\maketitle
\begin{abstract}
We prove a duality theorem applicable to a a wide range of  specialisations, as well as to some generalisations, of tangles in graphs. It generalises the classical tangle duality theorem of Robertson and Seymour, which says that every graph either has a large-order tangle or a certain low-width tree-decomposition witnessing that it cannot have such a tangle.%
   \COMMENT{}%
   \looseness=-1

Our result also yields duality theorems for profiles and for $k$-blocks. This solves a problem studied, but not solved, by Diestel and Oum and answers an earlier question of Carmesin, Diestel, Hamann and Hundertmark.%
   \COMMENT{}\looseness=-1
\end{abstract}

\section{Introduction}
There are a number of theorems about the structure of sparse graphs that assert a duality between the existence of a highly connected substructure and a tree-like overall structure. For example, if a connected graph $G$ has no 2-connected subgraph or minor, it is a tree. Less trivially, the graph minor structure theorem of Robertson and Seymour says that if $G$ has no $K_n$-minor then it has a \td\ into parts that are `almost' embeddable in a surface of bounded genus; see~\cite{DiestelBook16}.

Another example of a highly connected substructure is that of a \emph{$k$-block}, introduced by Mader~\cite{mader78} in 1978 and studied more recently in~\cite{ForcingBlocks, confing, CG14:isolatingblocks}. This is a maximal set of at least $k$ vertices in a graph~$G$ such that no two of them can be separated in~$G$ by fewer than $k$ vertices.

One of our main results is that the non-existence of a $k$-block, too, is always witnessed by a tree structure (Theorem~\ref{t:block}). This problem was raised in~\cite[Sec.~7]{ForcingBlocks}.%
   \COMMENT{}
   In~\cite{TangleTreeAbstract}, Diestel and Oum used a new theory of `abstract separation systems'~\cite{AbstractSepSys} in pursuit of this problem, but were unable to find the tree structures needed: the simplest witnesses to the nonexistence of $k$-blocks they could find are described in~\cite{DiestelOumDualityII}, but they are more complicated than trees. Our proof of Theorem~\ref{t:block}, and the rest of this paper, are still based on the theory developed in \cite{AbstractSepSys} and~\cite{TangleTreeAbstract}, and we show that there are tree-like obstructions to the existence of $k$-blocks after all.\looseness=-1

Tangles, introduced by Robertson and Seymour in~\cite{GMX}, are substructures of graphs that also signify high local connectivity, but of a less tangible kind than subgraphs, minors, or blocks. Basically, a tangle does not tell us `what' that substructure is, but only `where' it is: by orienting all the low-order separations of the graph in some consistent way, which we then think of as pointing `towards the tangle'. See Section~\ref{s:dual} below, or~\cite{DiestelBook16}, for a formal introduction to tangles.

Tangles come with dual tree structures called {\em branch decompositions\/}. Although defined differently, they can be thought of as \td s of a particular kind. In~\cite{TangleTreeAbstract, TangleTreeGraphsMatroids}, Diestel and Oum generalised the notion of tangles to ways of consistently orienting the low-order separations of a graph so as to describe other known types of highly connected substructures too, such as those dual to \td s of low width. We shall build on~\cite{TangleTreeAbstract} to find dual tree structures for various types of tangles, for blocks, and for `profiles': a common generalisation of blocks and tangles introduced in~\cite{ProfilesNew} and defined formally in Section~\ref{s:prof}.

Let us describe our results  more precisely. A~classical $k$-tangle, as in~\cite{GMX}, is an orientation of all the separations $\{A,B\}$ of order~${<k}$ in a graph~$G$, say as~$(A,B)$ rather than as~$(B,A)$, so that no three of these cover~$G$ by the subgraphs that $G$ induces on their `small sides'~$A$. Let $\T$ denote the set of all such forbidden triples of oriented separations of~$G$, irrespective of their order. The tangle duality theorem of Robertson and Seymour then asserts that, given~$k$, either the set $S_k$ of all the separations of~$G$ of order~$<k$ can be oriented in such a way as to induce no triple from~$\T$~-- an orientation of~$S_k$ we shall call a $\T$-tangle~-- or $G$ has a \td\ of a particular type: one from which it is clear that $G$ cannot have a $k$-tangle, i.e., a $\T$-tangle of~$S_k$.

Now consider any superset $\F$ of~$\T$. Given~$k$, the orientations of~$S_k$ with no subset in~$\F$ will then be particular types of tangles. Our $\F$-tangle duality theorem yields duality theorems for all these: if $G$ contains no such `special' tangle, it will have a \td\ that witnesses this. Formally, to every such~$\F$ and~$k$ there will correspond a class $\T_\F(k)$ of \td s that witness the non-existence of an $\F$-tangle of~$S_k$, and which are shown to exist whenever a graph has no $\F$-tangle of~$S_k$:

\begin{theorem}\label{t:Ftangle}
For every finite graph $G$, every set $\F\supseteq\T$ of sets of separations of~$G$, and every integer $k > 2$, exactly one of the following statements holds:
\begin{itemize}\itemsep=0pt\vskip-\smallskipamount\vskip0pt
\item $G$ admits an $\F$-tangle of~$S_k$;
\item $G$ has a \td\ in~$\T_\F(k)$.
\end{itemize}
\end{theorem}

Every $k$-block also defines an orientation of~$S_k$: as no separation ${\{A,B\}\in S_k}$ separates it, it lies entirely in~$A$ or entirely in~$B$. These orientations of~$S_k$ need not be $k$-tangles, so we cannot apply Theorem~\ref{t:Ftangle} to obtain a duality theorem for $k$-blocks. But still, we shall be able to define classes $\T_\B(k)$ of \td s that witness the non-existence of a $k$-block, in the sense that graphs with such a \td\ cannot contain one, and which always exist for graphs without a $k$-block:

\begin{theorem}\label{t:block}
For every finite graph $G$ and every integer $k > 0$ exactly one of the following statements holds:
\begin{itemize}\itemsep=0pt\vskip-\smallskipamount\vskip0pt
\item $G$ contains a $k$-block;
\item $G$ has a \td\ in~$\T_\B(k)$.
\end{itemize}
\end{theorem}

Finally, we define classes $\T_\P(k)$ of \td s which graphs with a $k$-profile cannot have, and prove the following duality theorem for profiles:

\begin{theorem}\label{t:prof}
For every finite graph $G$ and every integer $k > 2$ exactly one of the following statements holds:
\begin{itemize}\itemsep=0pt\vskip-\smallskipamount\vskip0pt
\item $G$ has a $k$-profile;
\item $G$ has a \td\ in~$\T_\P(k)$.
\end{itemize}
\end{theorem}

For readers already familiar with profiles~\cite{ProfilesNew} we remark that, in fact, we shall obtain a more general result than Theorem~\ref{t:prof}: our Theorem~\ref{t:main} is a duality theorem for all regular profiles in abitrary submodular%
   \COMMENT{}
   abstract separation systems, including the standard ones in graphs and matroids but many others too~\cite{MonaLisa}.

Like Theorem~\ref{t:Ftangle}, Theorems \ref{t:block} and~\ref{t:prof} are `structural' duality theorems in that they identify a structure that a graph~$G$ cannot have if it contains a $k$-block or $k$-profile, and must have if it does not. Alternatively, we can express the same duality more compactly in terms of graph invariants, as follows. Let
 \begin{eqnarray*}
  \beta(G) &:=& \max\,\{\,k\mid G\text{ has a $k$-block }\}\\
  \pi(G) &:=&\max\,\{\,k\mid G\text{ has a $k$-profile }\}
 \end{eqnarray*}
be the {\em block number\/} and the {\em profile number\/} of~$G$, respectively, and let
 \begin{eqnarray*}
   \bw(G) &:=& \min\,\{\,k\mid G\text{ has a \td\ in }\T_\B(k+1)\,\}\\
   \pw(G) &:=& \min\,\{\,k\mid G\text{ has a \td\ in }\T_\P(k+1)\,\}
 \end{eqnarray*}%
   \COMMENT{}
 be its {\em block-width\/} and {\em profile-width\/}. Theorems~\ref{t:block} and~\ref{t:prof} can now be rephrased as

\begin{corollary}\label{invariants}
As invariants of finite graphs, the block and profile numbers agree with the block- and profile-widths:
 $$\beta = \bw\quad\text{and}\quad\pi = \pw.\eqno\qed$$
\end{corollary}

In Section \ref{s:background} we introduce just enough about abstract separation systems~\cite{AbstractSepSys} to state the fundamental duality theorem of~\cite{TangleTreeAbstract}, on which all our proofs will be based. In Section \ref{s:proof} we give a proof of our main result, a duality theorem for regular profiles in submodular abstract separation systems. In Section~\ref{s:apply} we apply this to obtain structural duality theorems for $k$-blocks and $k$-profiles, and deduce Theorems \ref{t:Ftangle}--\ref{t:prof} as corollaries.
In Section~\ref{sec:width} we derive some bounds for the above width-parameters in terms of tree-width and branch-width.%
   \COMMENT{}

Any terms or notation left undefined in this paper are explained in~\cite{DiestelBook16}.

\section{Background Material}\label{s:background}
\subsection{Separation systems}
A \emph{separation} of a graph $G$ is a set $\{A,B\}$ of subsets of $V(G)$ such that $A \cup B = V$, and there is no edge of $G$ between $A \setminus B$ and $B \setminus A$. There are two \emph{oriented separations} associated with a separation, $(A,B)$ and $(B,A)$. Informally we think of $(A,B)$ as \emph{pointing towards} $B$ and \emph{away from} $A$. We can define a partial ordering on the set of oriented separations of $G$ by 
\[
(A,B) \leq (C,D) \text{ if and only if } A \subseteq C \text{ and } B \supseteq D.
\]
The \emph{inverse} of an oriented separation $(A,B)$ is the separation $(B,A)$, and we note that mapping every oriented separation to its inverse is an involution which reverses the partial ordering.

In~\cite{TangleTreeAbstract} Diestel and Oum generalised these properties of separations of graphs and worked in a more abstract setting. They defined a \emph{separation system} ${(\ra{S},\leq,{}^*)}$ to be a partially ordered set $\ra{S}$ with an order-reversing involution~${}^*$. The elements of $\ra{S}$ are called \emph{oriented separations}. Often a given element of~$\ra{S}$ is denoted by $\ra{s}$, in which case its inverse $\ra{s}^*$ will be denoted by $\la{s}$, and vice versa. Since ${}^*$ is ordering reversing we have that, for all $\ra{r},\ra{s} \in S$,
\[
\ra{r} \leq \ra{s} \text{ if and only if } \la{r} \geq \la{s}.\label{invcomp}
\]
A \emph{separation} is a set of the form $\{\ra{s},\la{s}\}$, and will be denoted by simply $s$. The two elements $\ra{s}$ and $\la{s}$ are the \emph{orientations} of $s$. The set of all such pairs $\{\ra{s},\la{s}\} \subseteq \ra{S}$ will be denoted by $S$.  If $\ra{s}=\la{s}$ we say $s$ is \emph{degenerate}. Conversely, given a set $S' \subseteq S$ of separations we write $\ra{S'} := \bigcup S'$ for the set of all orientations of its elements. With the ordering and involution induced from~$\ra{S}$, this will form a separation system.

Given a separation of a graph $\{A,B\}$ we can identify it with the pair $\{(A,B),(B,A)\}$ and in this way any set of oriented separations in a graph which is closed under taking inverses forms a separation system. When we refer to an oriented separation in a context where the notation explicitly indicates orientation, such as $\ra{s}$ or $(A,B)$, we will usually suppress the prefix ``oriented" to improve the flow of the narrative.

The \emph{separator} of a separation $s=\{A,B\}$ in a graph, and of its orientations~$\ra{s}$, is the set $A \cap B$. The \emph{order} of $s$ and $\ra{s}$, denoted as $|s|$ or as $|\ra{s}|$, is the cardinality of the separator, $|A \cap B|$.  Note that if $\ra{r}=(A,B)$ and $\ra{s}=(C,D)$ are separations then so are their \emph{corner separations} $\ra{r} \vee \ra{s} := (A \cup C, B \cap D)$ and $\ra{r} \wedge \ra{s} := (A \cap C , B \cup D)$. Our function $\ra{s}\mapsto|\ra{s}|$ is clearly {\em symmetric\/} in that $|\vs| = |\sv|$, and {\em submodular\/} in that 
\[
 |\ra{r} \vee \ra{s}| + |\ra{r} \wedge \ra{s}| \le |\ra{r}| + |\ra{s}|
\]
(in fact, with equality).

\begin{figure}[!ht]
\centering
\includegraphics[scale=1]{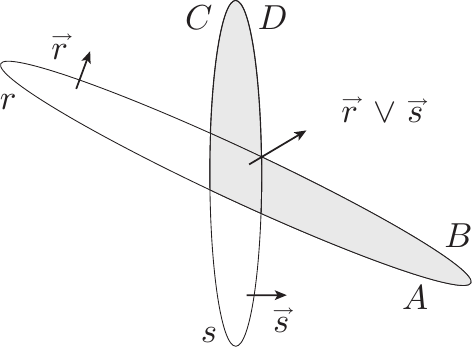}
\caption{The corner separation $\vr\vee\vs = (A \cup C, B \cap D)$}
\end{figure}

If an abstract separation system $(\vS,\le,{}^*)$ forms a lattice, i.e., if there exist binary operations $\vee$ and $\wedge$ on $\ra{S}$ such that $\ra{r} \vee \ra{s}$ is the supremum and $\ra{r} \wedge \ra{s}$ is the infimum of $\ra{r}$ and~$\ra{s}$, then we call $(\ra{S},\leq,{}^*,\vee,\wedge)$ a \emph{universe} of (oriented) separations. By~\eqref{invcomp}, it satisfies De~Mor\-gan's law:
\begin{equation}\label{deMorgan}
   (\vr\vee\vs)^* =\> \rv\wedge\sv.
\end{equation}%
   \COMMENT{}
   Any real, non-negative, symmetric and submodular function on a universe of separations, usually denoted as $\vs\mapsto |\vs|$, will be called an \emph{order function}.

Two separations $r$ and $s$ are \emph{nested} if they have $\leq$-comparable orientations. Two oriented separations $\ra{r}$ and $\ra{s}$ are \emph{nested} if $r$ and $s$ are nested.%
   \footnote{In general we will use terms defined for separations informally for oriented separations when the meaning is clear, and vice versa.}
    We say that $\ra{r}$ \emph{points towards $s$} (and $\la{r}$ \emph{points away from $s$}) if $\ra{r} \leq \ra{s}$ or $\ra{r} \leq \la{s}$. So two nested oriented separations are either $\leq$-comparable, or they point towards each other, or they point away from each other. If $\ra{r}$ and $\ra{s}$ are not nested we say that they \emph{cross}. A set of separations $S$ is \emph{nested} if every pair of separations in $S$ is nested, and a separation $s$ is \emph{nested} with a nested set of separations $S$ if $S \cup \{s\}$ is nested. 

A separation $\ra{r} \in \ra{S}$ is \emph{trivial in $\ra{S}$}, and $\la{r}$ is \emph{co-trivial}, if there exist an $s \in S$ such that $\ra{r} < \ra{s}$ and $\ra{r} < \la{s}$.  Note that if $\ra{r}$ is trivial, witnessed by some $s$, then, since the involution on~$\ra{S}$ is order-reversing, we have $\ra{r} < \ra{s} < \la{r}$. So, in particular, $\la{r}$ cannot also be trivial. Separations $\ra{s}$ such that $\ra{s} \leq \la{s}$, trivial or not, will be called \emph{small}.

In the case of separations of a graph $(V,E)$, the small separations are precisely those of the form $(A,V)$. The trivial separations are those of the form $(A,V)$ with $A \subseteq C \cap D$ for some separation $\{C,D\} \neq \{A,B\}$. Finally we note that there is only one degenerate separation in a graph, $(V,V)$.

\subsection{Tangle-tree duality in separation systems}\label{s:dual}

Let $\ra{S}$ be a separation system. An \emph{orientation} of $S$ is a subset $O \subseteq \ra{S}$ which for each $s \in S$ contains exactly one of its orientations $\ra{s}$ or~$\la{s}$. Given a universe $\ra{U}$ of separations with an order function, such as all the oriented separations of a given graph, we denote by 
\[
\ra{S_k} = \{ \ra{s} \in \ra{U} \colon  |\ra{s}| < k\}
\]
the set of all its separations of order less than~$k$. Note that $\ra{S_k}$ is again a separation system. But it is not necessarily a universe, since it may fail to be closed under the operations $\lor$ and~$\land$.

If we have some structure $\cc{C}$ in a graph that is `highly connected' in some sense, we should expect that no low order separation will divide it: that is, for every separation $s$ of sufficiently low order, $\cc{C}$ should lie on one side of~$s$ but not the other. Then $\cc{C}$ will {\em orient\/} $s$ as $\ra{s}$ or~$\la{s}$, choosing the orientation that `points to where it lies' according to some convention. For graphs, our convention is that the orientated separation $(A,B)$ points towards~$B$. And that if $\cc{C}$ is a $K_n$-minor of~$G$ with $n \geq k$, say, then $\cc{C}$ `lies on the side~$B$' if it has a branch set in~$B\setminus A$. (Note that it cannot have a branch set in $A\setminus B$ then.) Then $\cc{C}$ orients $\{A,B\}$ towards~$B$ by choosing $(A,B)$ rather than~$(B,A)$. In this way, $\cc{C}$ induces an orientation of all of~$S_k$. 

The idea of~\cite{TangleTreeAbstract}, now, following the idea of tangles, was to \emph{define} `highly connected substructures' in this way: as orientations of a given set $S$ of separations.

Any concrete example of `highly-connected substructures' in a graph, such as a $K_n$-minor or a $k$-block, will not induce arbitrary orientations of $S_k$: these orientations will satisfy some consistency rules. For example, consider two separations $(A,B) < (C,D)$. If our `highly connected' structure $\cc{C}$ orients $\{C,D\}$ towards $D$ then, since $B \supseteq D$ it should not orient $\{A,B\}$ towards~$A$.

We call an orientation $O$ of a set $S$ of separations in some universe~$\ra{U}$ \emph{consistent} if whenever we have distinct $r$ and $s$ such that $\ra{r} < \ra{s}$, the set $O$ does not contain both $\la{r}$ and $\ra{s}$. Note that a consistent orientation of~$S$ must contain all separations $\ra{r}$ that are trivial in~$S$ since, if $\ra{r} < \ra{s}$ and $\ra{r} < \la{s}$, then $\la{r}$ would be inconsistent with whichever orientation of $s$ lies in~$O$.

Given a set $\cc{F}$, we say that an orientation $O$ of $S$ \emph{avoids} $\cc{F}$ if there is no $F \in \cc{F}$ such that $F \subseteq O$.  So for example an orientation of $S$ is consistent if it avoids $\cc{F} = \{ \{\la{r},\ra{s}\} \subseteq \ra{S} \colon  r \neq s, \ra{r} < \ra{s} \}$. In general we will define the highly connected structures we consider by the collection $\cc{F}$ of subsets they avoid. For example a \emph{tangle of order $k$}, or $k${\em -tangle}, in a graph $G$ is an orientation of $S_k$ which avoids the set of triples 
\begin{equation}\label{tangleaxiom}
\cc{T} = \{\{(A_1,B_1),(A_2,B_2),(A_3,B_3)\} \subseteq \ra{U} \,: \, \bigcup_{i=1}^3 G[A_i] = G\}.
\end{equation}
Here, the three separations need not be distinct, so any $\cc{T}$-avoiding orientation of~$S_k$ will be consistent. More generally, we say that a consistent orientation of a set $S$ of separations which avoids some given set $\cc{F}$ is an \emph{$\cc{F}$-tangle (of~$S$)}. 

Given a set $S$ of separations, an \emph{$S$-tree} is a pair $(T,\alpha)$, of a tree $T$ and a function $\alpha : \overrightarrow{E(T)} \rightarrow \ra{S}$ from the set $\overrightarrow{E(T)}$ of directed edges of $T$ such that
\begin{itemize}
\item For each edge $(t_1,t_2) \in \overrightarrow{E(T)} $, if $\alpha(t_1,t_2) = \ra{s}$ then $\alpha(t_2,t_1) = \la{s}$.
\end{itemize}
The $S$-tree is said to be {\em over\/} a set~$\cc{F}$ if
\begin{itemize}
\item For each vertex $t \in T$, the set $\{\alpha(t',t) \, :\, (t',t) \in \overrightarrow{E(T)} \}$ is in $\cc{F}$.
\end{itemize}

Particularly interesting classes of $S$-trees are those over sets~$\cc{F}$ of `stars'. A~set $\sigma$ of nondegenerate oriented separations is a \emph{star} if $\ra{r} \leq \la{s}$ for all distinct ${\ra{r},\ra{s} \in \sigma}$. We say that a set $\cc{F}$ \emph{forces} a separation $\ra{r}$ if $\{ \la{r} \} \in \cc{F}$. And $\cc{F}$ is \emph{standard} if it forces every trivial separation in $\ra{S}$.

The main result of~\cite{TangleTreeAbstract} asserts a duality between $S$-trees over~$\cc{F}$ and $\cc{F}$-tangles when $\cc{F}$ is a standard set of stars satisfying a certain closure condition. Let us describe this next.

Suppose we have a separation $\ra{r}$ which is neither trivial nor degenerate. Let $S_{\geq \ra{r}}$ be the set of separations $x \in S$ that have an orientation $\ra{x} \geq \ra{r}$. Given $x \in S_{\geq \ra{r}} \setminus \{r\}$ we have, since $\ra{r}$ is nontrivial, that only one of the two orientations of $x$, say $\ra{x}$, is such that $\ra{x} \geq \ra{r}$ and $x$ is not degenerate. For every $\ra{s}\geq \ra{r}$ we can define a function $f\!\!\downarrow^{\ra{r}}_{\ra{s}}$ on $\ra{S}_{\geq \ra{r}} \setminus \{\la{r}\}$ by%
   \footnote{The exclusion of $\la{r}$ here is for a technical reason: if $\ra{r} < \la{r}$, we do not want to define $f\!\!\downarrow^{\ra{r}}_{\ra{s}}(\la{r})$ explicitly, but implicitly as the inverse of $f\!\!\downarrow^{\ra{r}}_{\ra{s}}(\ra{r})$. }
\[
f\!\!\downarrow^{\ra{r}}_{\ra{s}}(\ra{x}) := \ra{x} \vee \ra{s} \text{ and } f\!\!\downarrow^{\ra{r}}_{\ra{s}}(\la{x}) := (\ra{x} \vee \ra{s})^*.
\]
In general, the image in~$\ra{U}$ of this function need not lie in~$\ra{S}$.

\begin{figure}[!ht]
\centering
\includegraphics[scale=1]{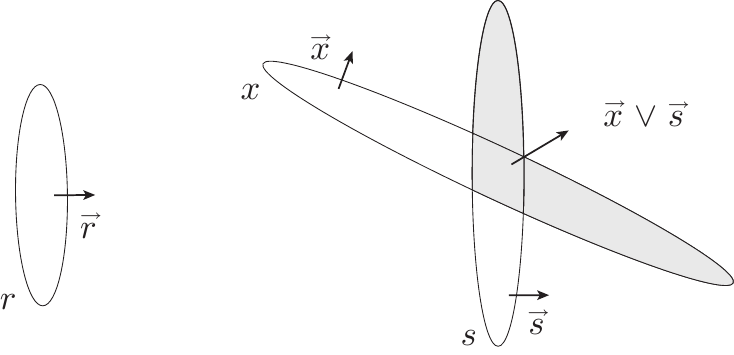}
\caption{Shifting a separation $\ra{x} \geq \ra{r}$ to $f\!\!\downarrow^{\ra{r}}_{\ra{s}}(\ra{x}) = \ra{x}\lor \ra{s}$. }
\end{figure}

We say that $\ra{s}\in\vS$ \emph{emulates $\ra{r}\in\vS$ in~$\vS$} if $\ra{r} \leq \ra{s}$ and the image of $f\!\!\downarrow^{\ra{r}}_{\ra{s}}$ is contained in $\ra{S}$. Given a standard set $\cc{F}$ of stars, we say further that $\ra{s}$ \emph{emulates}~$\ra{r}$ in~$\vS$ {\em for}~$\cc{F}$ if $\ra{s}$ emulates~$\ra{r}$ in~$\vS$ and the image under $f\!\!\downarrow^{\ra{r}}_{\ra{s}}$ of every star $\sigma \subseteq \ra{S}_{\geq \ra{r}} \setminus \{\la{r}\}$ that contains some separation $\ra{x}$ with $\ra{x} \geq \ra{r}$ is again in $\cc{F}$. 

We say that a separation system $\ra{S}$ is \emph{separable} if for any two nontrivial and nondegenerate separations $\ra{r},\la{r'} \in \ra{S}$ such that $\ra{r} \leq \ra{r'}$ there exists a separation $s \in S$ such that $\ra{s}$ emulates~$\ra{r}$ in~$\vS$ and $\la{s}$ emulates~$\la{r'}$ in~$\vS$. We say that $\ra{S}$ is \emph{$\cc{F}$-separable} if for all nontrivial and nondegenerate $\ra{r},\la{r'} \in \ra{S}$ that are not forced by $\cc{F}$ and such that $\ra{r} \leq \ra{r'}$ there exists a separation $s \in S$ with an orientation $\ra{s}$ that emulates~$\ra{r}$ in~$\vS$ for~$\cc{F}$ and such that $\la{s}$ emulates~$\la{r'}$ in~$\vS$ for~$\cc{F}$. Often one proves that $\ra{S}$ is $\cc{F}$-separable in two steps, by first showing that it is separable, and then showing that $\cc{F}$ is \emph{closed under shifting}: that whenever $\ra{s}$ emulates (in~$\vS$) some nontrivial and nondegenerate $\ra{r}$ not forced by~$\cc{F}$, then it does so for~$\cc{F}$.

We are now in a position to state the Strong Duality Theorem from~\cite{TangleTreeAbstract}.

\begin{theorem}\label{t:dual}
Let $\ra{S}$ be a separation system in some universe of separations, and $\cc{F}$ a standard set of stars. If $\ra{S}$ is $\cc{F}$-separable, exactly one of the following assertions holds:
\begin{itemize}
\item There exists an $S$-tree over $\cc{F}$;
\item There exists an $\cc{F}$-tangle of $S$.
\end{itemize}
\end{theorem}

The property of being $\cc{F}$-separable may seem a rather strong condition. However in~\cite{TangleTreeGraphsMatroids} it is shown that for every graph the set $\ra{S_k}$ is separable, and all the sets $\cc{F}$ of stars whose exclusion describes classical notions of highly connected substructures are closed under shifting. Hence in all these cases $\ra{S_k}$ is $\cc{F}$-separable, and Theorem~\ref{t:dual} applies.

One of our main tasks will be to extend the applicability of Theorem~\ref{t:dual} to sets $\cc{F}$ of separations that are not stars, by constructing a related set $\cc{F}^*$ of stars whose exclusion is tantamount to excluding~$\cc{F}$. 

\subsection{Blocks, tangles, and profiles}\label{s:prof}
Suppose we have a graph $G=(V,E)$ and are considering the set $U$ of its separations. As mentioned before, it is easy to see that the tangles of order~$k$ in~$G$, as defined by Robertson and Seymour~\cite{GMX}, are precisely the $\cc{T}$-tangles of $S_k = \{ s \in U \colon  |s|<k\}$. In this case, if we just consider the set of stars in $\cc{T}$,

\medskip

 $\cc{T}^* :=  \big\{\,\{(A_1,B_1),(A_2,B_2),(A_3,B_3)\} \subseteq \ra{S} :$

   \smallskip\indent\indent\indent\indent\indent\indent
  $(A_i,B_i) \leq (B_j,A_j) \text{ for all } i,j \text{ and } \bigcup_{i=1}^3 G[A_i] = G\,\big\},$

   \medskip\noindent
then the $\cc{T}^*$-tangles of $S_k$ are precisely its $\cc{T}$-tangles~\cite{TangleTreeGraphsMatroids}. That is, a consistent orientation of $S_k$ avoids $\cc{T}$ if and only if it avoids $\cc{T}^*$. It is a simple check that $\cc{T}^*$ is a standard set of stars which is closed under shifting, and hence Theorem~\ref{t:dual} tells us that every graph either has a tangle of order $k$ or an $S_k$-tree over $\cc{T}^*$, but not both.%
   \COMMENT{}

Another highly connected substructure that has been considered recently in the literature are $k$-blocks. Given $k \in \bb{N}$ we say a set $I$ of at least $k$ vertices in a graph $G$ is \emph{$(<\!\!k)$-inseparable} if no set $Z$ of fewer than $k$ vertices separates any two vertices of $I \setminus Z$ in $G$. A maximal $(<\!\!k)$-inseparable set of vertices is called a \emph{$k$-block}. These objects were first considered by Mader~\cite{mader78}, but have been the subject of recent research~\cite{ForcingBlocks, confing, CG14:isolatingblocks}.

As indicated earlier, every $k$-block $b$ of $G$ defines an orientation $O(b)$ of~$S_k$: 
 $$O(b) := \{ (A,B) \in {\ra{S_k}} \colon  b \subseteq B\}.$$
Indeed, for each separation $\{A,B\} \in S_k$ exactly one of $(A,B)$ and $(B,A)$ will be in $O(b)$, since $A \cap B$ is too small to contain~$b$ and does not separate any two of its vertices. Hence, $O(b)$ is indeed an orientation of~$S_k$. Note also that $O(b)\ne O(b')$ for distinct $k$-blocks $b\ne b'$: by their maximality as $k$-indivisible sets of vertices there exists a separation $\{A,B\}\in S_k$ such that $A\setminus B$ contains a vertex of~$b$ and $B\setminus A$ contains a vertex of~$b'$, which implies that $(A,B)\in O(b')$ and $(B,A)\in O(b)$.

The orientations $O(b)$ of $S_k$ defined by a $k$-block $b$ clearly avoid
\[
\cc{B}_k := \big\{ \{ (A_i,B_i) \colon  i \in I\} \subseteq \ra{U} \colon  | \bigcap_{i \in I} B_i | < k\big\},
\]
since $b \subseteq B_i$ for every $(A_i,B_i) \in O(b)$ and $|b|\geq k$. Also, it is easily seen that every $O(b)$ is consistent. Thus, every such orientation $O(b)$ is an $\cc{F}$-tangle of $S_k$ for $\cc{F} = \cc{B}_k$. Conversely, if $O \subseteq S_k$ is a $\cc{B}_k$-tangle of $S_k$, then $b := \bigcap {\{ B \,|\, (A,B) \in O \}}$ is easily seen to be a $k$-block, and $O = O(b)$. The orientations of $S_k$ that are defined by a $k$-block, therefore, are precisely its $\cc{B}_k$-tangles.

The $\cc{B}_k$-tangles of $S_k$ and its $\cc{T}$-tangles (i.e., the ordinary $k$-tangles of~$G$) share the property that if they contain separations $(A,B)$ and $(C,D)$, then they cannot contain the separation $(B \cap D, A \cup C)$. Indeed, clearly this condition is satisfied by $O(b)$ for any $k$-block $b$, since if $b \subseteq B$ and $b \subseteq D$ then $b \subseteq B \cap D$ and hence $b \not\subseteq A \cup C$ if $\{B\cap D, A\cup C\}\in S_k$. For tangles, suppose that some tangle contains such a triple $\{ (A,B),(C,D),(B \cap D, A \cup C) \}$. Since $\{A,B\}$ and $\{C,D\}$ are separations of~$G$, every edge not contained in $G[A]$ or $G[C]$ must be in $G[B]$ and $G[D]$, and hence in $G[B \cap D]$. Therefore $G[A] \cup G[C] \cup G[B \cap D] = G$, contradicting the fact that the tangle avoids~$\cc{T}$.

Informally, if we think of the side of an oriented separation to which it points as `large', then the orientations of~$S_k$ that form a tangle or are induced by a $k$-block have the natural property that if $B$ is the large side of $\{A,B\}$ and $D$ is the large side of $\{C,D\}$ then $B \cap D$ should be the large side of $\{A \cup C, B \cap D\}$~-- if this separation is also in~$S_k$, and therefore oriented by~$O$. That is, the largeness of separation sides containing blocks or tangles is preserved by taking intersections.

Consistent orientations with this property are known as `profiles'. Formally, a \emph{$k$-profile in~$G$} is a $\cc{P}$-tangle of $S_k$ where
\[
\cc{P} := \big\{ \sigma \subseteq \ra{U} \mid\exists\, A,B,C,D \subseteq V\colon \sigma= \{ (A,B),(C,D),(B \cap D, A \cup C) \}\big\}.
\]

As we have seen,

\begin{lemma}
All orientations of $S_k$ that are tangles, or of the form~$O(b)$ for some $k$-block $b$ in~$G$, are $k$-profiles in~$G$.\qed
\end{lemma}

We remark that, unlike in the case of $\cc{T}$, the subset $\cc{P}'$ consisting of just the stars in~$\cc{P}$ yields a wider class of tangles: there are $\cc{P}'$-tangles of $S_k$ that are not $\cc{P}$-tangles, i.e., which are not $k$-profiles.

More generally, if $\ra{S}$ is any separation system contained in some universe~$\ra{U}$, we can define a {\em profile of $S$} to be any $\cc{P}$-tangle of $S$ where
\[
\cc{P} := \big\{ \sigma \subseteq \ra{U} \mid \exists\, \ra{r}, \ra{s} \in \ra{U} \colon\sigma= \{ \ra{r},\ra{s}, \la{r} \wedge \la{s} \} \big\}.
\]
In particular, all $\cc{F}$-tangles with $\cc{F}\supseteq\cc{P}$ will be profiles.

The initial aim of Diestel and Oum in developing their duality theory~\cite{TangleTreeAbstract} had been to find a duality theorem broad enough to imply duality theorems for $k$-blocks and $k$-profiles. Although their theory gave rise to a number of unexpected results~\cite{TangleTreeGraphsMatroids}, a~duality theorem for blocks and profiles was not among these; see~\cite{DiestelOumDualityII} for a summary of their findings on this problem.

Our next goal is to show that their Strong Duality Theorem does implies duality theorems for blocks and profiles after all.

\section{A duality theorem for abstract profiles}\label{s:proof}
In this section we will show that Theorem \ref{t:dual} can be applied to many more types of profiles than originally thought. These will include both $k$-profiles and $k$-blocks in graphs.

We say that a separation system in some universe%
   \footnote{Although submodular separation systems $\ra{S}$ have to lie in some universe $\ra{U}$ in order for $\wedge$ and $\vee$ to be defined on $\ra{S}$ (but with images that may lie in $\ra{U} \setminus \ra{S}$), the choice of $\ra{U}$, given $\ra{S}$, will not matter to us. We shall therefore usually introduce submodular separation systems $\ra{S}$ without formally introducing such a universe $\ra{U} \supseteq \ra{S}$.}
   is \emph{submodular} if for every two of its elements $\ra{r},\ra{s}$ it also contains at least one of $\ra{r} \wedge \ra{s}$ and $\ra{r} \vee \ra{s}$. Given any (submodular) order function on a universe, the separation system
\[
{\ra{S_k}} = \{\ra{r} \in \ra{U} \colon  |\ra{r}| <k \}
\]
is submodular for each $k$. In particular, for any graph $G$, its universe $\ra{U}$ of separations and, for any integer $k \geq 1$, the separation system ${\ra{S_k}}$, is submodular.

We say that a subset $O$ of $U$ is \emph{strongly consistent} if it does not contain both $\la{r}$ and~$\ra{s}$ for any $\ra{r},\ra{s}\in\vS$ with $\vr < \vs$ (but not necessarily $r\ne s$, as in the definition of `consistent'). An orientation $O$ of~$S$, therefore, is strongly consistent if and only if for every $\vs\in O$ it also contains every $\vr\le\vs$ with $r\in S$. In particular, then, $O$~cannot contain any $\vs$ such that $\sv\le\vs$ (i.e., with $\sv$ is small).

Let us call an orientation $O$ of~$S$ {\em regular\/} if it contains all the small separations in~$\vS$.

\begin{lemma}\label{l:small}
An orientation $O$ of a separation system~$\vS$ is strongly consistent if and only if it is consistent and regular.
\end{lemma}

\begin{proof}
Clearly every strongly consistent orientation $O$ is also consistent. Suppose some small $\ra{s}\in\vS$ is not in~$O$. Then $\sv\in O$, since $O$ is an orientation of~$S$. Thus, $\ra{s} < \la{s}\in O$. But this implies $\vs\in O$, since $O$ is strongly consistent, contradicting the choice of~$\vs$. Hence $O$ contains every small separation.

Conversely, suppose $O$ is a consistent orientation of $S$ that is not strongly consistent. Then $O$ contains two distinct oriented separations $\rv$ and $\ra{s}$ such that $\ra{r} < \ra{s}$ and $r=s$.%
   \COMMENT{}
   Thus, $\sv = \vr < \vs$ is small but not in~$O$, as $\vs\in O$. Hence $O$~does not contain all small separations in~$\vS$.
\end{proof}

For example, the $\cc{B}_k$-tangles of~$S_k$ in a graph, as well as its ordinary tangles of order~$k$, are regular by Lemma~\ref{l:small}: they clearly contain all small separations in ${\ra{S_k}}$, those of the form $(A,V)$ with $|A|<k$, since they cannot contain their inverses~$(V,A)$. More generally:

\begin{lemma}\label{l:strong}
For $k > 2$, every $k$-profile in a graph $G=(V,E)$ is regular.
\end{lemma}

\begin{proof}
We have to show that every $k$-profile $O$ in $G$ contains every small separation in ${\ra{S_k}}$. Recall that these are precisely the separations $(A,V)$ of $G$ such that $|A| < k$.

Suppose first that $|A| < k-1$. Let $A'$ be any set such that $|A'|=k-1$ and $A \subset A'$. Then $\{A',V\} \in S_k$, and $(A,V) < (A',V)$ as well as ${(A,V) < (V,A')}$. Since $O$ contains $(A',V)$ or~$(V,A')$, its consistency implies that it also contains~$(A,V)$.

If $|A|=k-1$ then, since $k > 2$, we can pick two non-empty sets $A', A'' \subsetneq A$ such that $A' \cup A'' = A$. Since $|A'|,|A''|<k-1$, by the preceding discussion both $(A',V)$ and $(A'',V)$ lie in $O$. As $(V,A) \wedge (V,A'') = (V \cap V,A' \cup A'') = (V,A)$ and $\{ (A',V),(A'',V),(V,A)\} \in \cc{P}$, the fact that $O$ is a profile implies that $(V,A)\notin O$, so again $(A,V)\in O$ as desired.
\end{proof}

There can be exactly one irregular 1-profile in a graph $G = (V,E)$, and only if $G$ is connected: the set $\{(V,\emptyset)\}$.

Graphs can also have irregular $2$-profiles, but they are easy to characterise. Indeed, consider a $2$-profile $O$ and small separation $(\{x\},V)$. Suppose first that $x$ is a cutvertex of~$G$, in the sense that there exists some $\{A,B\} \in S_2$ such that $A \cap B = \{x\}$ and neither $A$ nor~$B$ equals~$V$.%
   \COMMENT{}
   Then $(\{x\},V) < (A,B)$ and $(\{x\},V) < (B,A)$, so the consistency of $O$ implies that $(\{x\},V) \in O$.

Therefore, if $(V,\{x\}) \in O$ then $x$ is not a cutvertex of $G$. Then, for every other separation $\{A,B\}$, either $x \in A \setminus B$ or $x \in B \setminus A$, and so either $(B,A) < (V,\{x\})$ or $(A,B) < (V,\{x\})$. The consistency of $O$ then determines that 
\[
O = O_x := \{(A,B)\in \ra{S_2} \colon  x \in B \text{ and } (A,B) \neq (\{x\},V) \},
\]
which is indeed a profile.

We have shown that every graph contains, for each of its vertices $x$ that is not a cutvertex, a unique $2$-profile $O_x$ that is not strongly consistent. However, the orientation
\[
O'_x := \{(A,B)\in \ra{S_2} \colon  x \in B \text{ and } (A,B) \neq (V,\{x\}) \}
\]
of $S_2$ is also a $2$-profile which does contain every small separation in $\vS_{\!2}$.%
   \COMMENT{}
   (Indeed, $O'_x = O(b)$ for the unique block~$b$ containing $x$.) Since every graph contains a vertex which is not a cutvertex, it follows that

\begin{lemma}\label{l:two}
Every graph $G$ contains a regular $2$-profile.\qed
\end{lemma}

Lemma~\ref{l:two} means that our goal to find a duality theorem for $k$-profiles in graphs has substance only for~$k > 2$, for which Lemma~\ref{l:strong} tells that all $k$-profiles are regular. In our pursuit of Theorems \ref{t:block} and~\ref{t:prof} it will therefore suffice to study regular $\F$-tangles of submodular separation systems~$\vS$, such as $\ra{S_k}$ for $\F=\P$.

So, until further notice: 
\begin{equation*}
  \begin{minipage}[c]{0.75\textwidth}\em
    Let $\vS$ be any submodular separation system in some universe~$\ra{U}$, and let $\F$ be a subset of~$2^{\ra{S}}$ containing $\cc{P} \cap 2^{\ra{S}}$.
  \end{minipage}
\end{equation*}
Our aim will be to prove a duality theorem for the regular $\cc{F}$-tangles of~$S$. 

It will be instructive to keep in mind, as an example, the case of $k$-blocks, where $\cc{F} = \cc{B}_k$. In this case any triple $\{(A,B),(C,D),(D \cap B, A \cup C)\} \in \cc{P}\cap 2^{\ra{S_k}}$ is contained in~$\cc{B}_k$, as $|B \cap D \cap (A \cup C)| < k$ since $(D \cap B, A \cup C) \in S_k$.

For ease of notation, let us write $\cc{P}_S := \cc{P} \cap 2^{\ra{S}}$, and put $\P_k := \P_{S_k}$ when $U$ is the set of separations of a given graph. Note that an orientation of $S$ avoids $\cc{P}$ if and only if it avoids~$\cc{P}_S$, and an $S$-tree is over $\cc{P}$ if and only if it is over $\cc{P}_S$.

Our first problem is that, in order to apply Theorem \ref{t:dual}, we need $\cc{F}$ to be a set of stars. Since our assumptions about $\cc{F}$ do not require this, our first aim is to turn $\cc{F}$ into a set $\cc{F}^*$ of stars such that the regular $\cc{F}$-tangles of $S$ are precisely its regular $\cc{F}^*$-tangles.

Suppose we have some pair of separations $\ra{x_1}$ and $\ra{x_2}$ which are both contained in some set $\sigma \subseteq \ra{S}$. Since $S$, by assumption, is submodular, at least one of $\ra{x_1} \wedge \la{x_2}$ and $\ra{x_2} \wedge \la{x_1}$ must also be in $\ra{S}$. To \emph{uncross $\ra{x_1}$ and $\ra{x_2}$ in $\sigma$} we replace $\{\ra{x_1},\ra{x_2}\}$ with the pair $\{\ra{x_1} \wedge \la{x_2},\ra{x_2}\}$ in the first case and $\{\ra{x_1},\ra{x_2} \wedge \la{x_1}\}$ in the second case. We note that, in both cases the new pair forms a star and is pointwise $\leq$ the old pair $\{ \ra{x_1}, \ra{x_2}\}$. Uncrossing every pair of separations in $\sigma$ in turn, we can thus turn $\sigma$ into a star $\sigma^*$ of separations in at most ${ |\sigma| \choose 2}$ steps, since any star of two separations remains a star if one of its elements is replaced by a smaller separation, and a set of oriented separations is a star as soon as all its $2$-subsets are stars. Note, however, that $\sigma^*$ will not in general be unique, but will depend on the order in which we uncross the pair of separations in $\sigma$. Let us say that $\cc{F}^*$ is an \emph{uncrossing} of a set  $\cc{F}$ of sets $\sigma \subseteq \ra{S}$ if
\begin{itemize}
\item Every $\tau \in \cc{F}^*$ can be obtained by uncrossing a set $\sigma \in \cc{F}$;
\item For every $\sigma \in \cc{F}$ there is some $\tau \in \cc{F}^*$ that can be obtained by uncrossing~$\sigma$\rlap.
\end{itemize}
Note that~$\F^*$, like $\F$, is a subset of~$2^{\vS}$. Also, $\F^*$~contains all the stars from~$\cc{F}$, since these have no uncrossings other than themselves. In particular, if $\cc{F}$ is standard, i.e. contains all the singleton stars $\{ \la{r}\}$ with $\ra{r}$ trivial in $\ra{S}$, then so is $\cc{F}^*$.

We have shown the following:
\begin{lemma}\label{l:uncross}
$\cc{F}$ has an uncrossing $\cc{F}^*$. If $\cc{F}$ is standard, then so is $\cc{F}^*$.
\end{lemma}
The smaller we can take $\cc{F}^*$ to be, the smaller will be the class of $S$-trees over $\cc{F}^*$. However, to make $\cc{F}^*$ as small as possible we would have to give it exactly one star $\tau$ for each $\sigma \in \cc{F}$, which would involve making a non-canonical choice with regards to the order in which we uncross $\sigma$, and possibly which of the two potential uncrossings of a given pair of separations we select.

If we wish for a more canonical choice of family, we can take $\cc{F}^*$ to consist of every star that can be obtained by uncrossing a set in $\cc{F}$ in any order. Obviously, this will come at the expense of increasing the class of $S$-trees over $\cc{F}$, i.e. the class of dual objects in our desired duality theorem. 

\begin{lemma}\label{l:*tangle}
Let $\cc{F}^*$ be an uncrossing of~$\cc{F}$. Then an orientation $O$ of $S$ is a regular $\cc{F}$-tangle if and only if it is a regular $\cc{F}^*$-tangle.
\end{lemma}
\begin{proof}
Let us first show that if $O$ is a regular $\cc{F}^*$-tangle then it is a regular $\cc{F}$-tangle. It is sufficient to show that $O$ avoids $\cc{F}$. Suppose for a contradiction that there is some $\sigma = \{\ra{x_1},\ra{ x_2}, \ldots, \ra{x_n}\} \in \cc{F}$ such that $\sigma \subseteq O$. Since $\cc{F}^*$ is an uncrossing of $\cc{F}$ there is some $\tau = \{\ra{u_1},\ra{u_2}, \ldots, \ra{u_n}\}  \in \cc{F}^*$ that is an uncrossing of $\sigma$. Then, $\ra{u_i} \leq \ra{x_i} \in O$ for all $i$. Since $O$ is strongly consistent, by Lemma~\ref{l:small}, this implies $\ra{u_i} \in O$ for each $i$. Therefore $\tau \subseteq O$, contradicting the fact that $O$ avoids $\cc{F}^*$.

Conversely suppose $O$ is a regular $\cc{F}$-tangle. We would like to show that $O$ avoids $\cc{F}^*$. To do so, we will show that, if $O$ avoids some set $\sigma = \{\ra{x_1},\ra{x_2}, \ldots , \ra{x_n}\}$ then it also avoids the set $\sigma' = \{\ra{x_1} \wedge \la{x_2},\ra{x_2} \ldots , \ra{x_n}\}$ obtained by uncrossing the pair $\ra{x_1},\ra{x_2}$. Then by induction $O$ must also avoid every star obtained by uncrossing a set in $\cc{F}$, and thus will avoid $\cc{F}^*$.

Suppose then that $O$ avoids $\sigma$ but $\sigma' \subseteq O$. Since $x_1 \in S$ either $\ra{x_1}$ or $\la{x_1}$ lies in $O$. As $\sigma\setminus \{\ra{x_1}\} = \{ \ra{x_2}, \ra{x_3}, \ldots , \ra{x_n}\} \subseteq \sigma' \subseteq O$, but $O$ avoids $\sigma$, we have $\ra{x_1} \not\in O$ and hence $\la{x_1} \in O$. But then $O$ contains the triple $\{\la{x_1}, \ra{x_2}, \ra{x_1} \wedge \la{x_2} \} \in \cc{P}_S \subseteq \cc{F}$. This contradicts the fact that $O$ avoids $\cc{F}$.
\end{proof}

Lemma~\ref{l:*tangle} has an interesting corollary. Suppose, that, in a graph, every star of separations in some given consistent orientation~$O$ of ${S_k}$ points to some $k$-block. Is there {\em one\/} $k$-block to which all these stars~-- and hence every separation in~$O$~-- point?%
   \COMMENT{}
   This is indeed the case:

\begin{corollary}
If every star of separations in some strongly consistent orientation of~$S$ is contained in some profile of~$S$, then there exists one profile of~$S$ that contains all these stars.
\end{corollary}

\begin{proof}
In Lemma~\ref{l:*tangle}, take $\F:= \P_S$. An orientation $O$ of~$S$ whose stars each lie in a profile of~$S$ cannot contain a star from~$\P^*_S$. But if $O$ is regular, consistent, and avoids~$\P^*_S$, then by Lemma~\ref{l:*tangle} it also avoids~$\P_S$ and hence is a $\P_S$-tangle.%
   \COMMENT{}
\end{proof}

Before we can apply Theorem \ref{t:dual} to our newly found set $\cc{F}^*$ of stars, we have to overcome another problem: $\ra{S}$ may fail to be $\cc{F}^*$-separable. To address this problem, let us briefly recall what it means for a family to be closed under shifting. Suppose we have a a pair of separations $\ra{r} \leq \ra{s}$ such that $\ra{r}$ is nontrivial, nondegenerate, and not forced by~$\F$. Suppose further that $\ra{s}$ emulates~$\ra{r}$ in~$\vS$, and that we have a star $\tau = \{\ra{x_1}, \ra{x_2}, \ldots , \ra{x_n} \} \subseteq \ra{S}_{\geq \ra{r}} \setminus \{\la{r}\}$ that contains some separation $\ra{x_1} \geq \ra{r}$. Then the image $\tau'$ of $\tau$ under $f\!\!\downarrow^{\ra{r}}_{\ra{s}}$ is
\[
\tau' = \{ \ra{x_1} \vee \ra{s}, \ra{x_2} \wedge \la{s}, \ldots, \ra{x_n} \wedge \la{s}\},
\]
where the fact that $\ra{s}$ emulates~$\ra{r}$ guarantees that $\tau' \subseteq \ra{S}$. Let us call $\tau'$ a \emph{shift of $\tau$}, and more specifically the \emph{shift of $\tau$ from $\ra{r}$ to $\ra{s}$}. (See~\cite{TangleTreeAbstract} for why this is well defined.)

%We note that, for certain stars $\tau$ it may be the case that there is more than one $\ra{x_i} \in \tau$ with $\ra{x_i} \geq \ra{r}$. However, suppose $\ra{r} \leq \ra{x_i}$ and $\ra{r} \leq \ra{x_j}$. Since $\tau$ is a star we have that $\ra{r} \leq \ra{x_i} \leq \la{x_j}$, and so, since $\ra{r}$ is nontrivial, $\ra{r}$ is equal to $\ra{x_j}$ or $\la{x_j}$. However, $\la{r} \not\in \tau$, and so $\ra{r} = \ra{x_j} = \ra{x_i}$ by symmetry. Therefore, given such $\ra{r}$, $\ra{s}$ and $\tau$ the shift of $\tau$ from $\ra{r}$ to $\ra{s}$ is well defined.

For a family $\cc{F}$ to be closed under shifting it is sufficient that it contains all shifts of its elements: that for every $\tau \in \cc{F}$, every $\ra{x_1} \in \tau$, every nontrivial and nondegenerate $\ra{r} \leq \ra{x_1}$ not forced by~$\F$, and every $\ra{s}$ emulating~$\ra{r}$, the shift of $\tau$ from $\ra{r}$ to $\ra{s}$ is in $\cc{F}$.

The idea for making Theorem \ref{t:dual} applicable to $\cc{F}^*$ will be to close $\cc{F}^*$ by adding any missing shifts. Let us define a family $\hat{\cc{F}^*}$ as follows: Let $\cc{G}_0 = \cc{F}^*$, define $\cc{G}_{n+1}$ inductively as the set of shifts of elements of $\cc{G}_n$, and put $\hat{\cc{F}^*} := \bigcup_n \cc{G}_n$. Clearly $\hat{\cc{F}^*}$ is closed under shifting.

Next, let us show that a strongly consistent orientation of~$S$ avoids $\cc{F}^*$ if and only if it avoids $\hat{\cc{F}^*}$. We first note the following lemma.

\begin{lemma}\label{l:shift}
Let $O$ be a regular $\cc{P}$-tangle of $S$.%
   \COMMENT{}
   Let $\sigma \subseteq \ra{S}$ be a star, and let $\sigma'$ be a shift of $\sigma$ from some $\ra{r}$ to some $\ra{s} \in \ra{S}$. Then $\sigma' \subseteq O$ implies that $\sigma \subseteq O$.
\end{lemma}
\begin{proof}
Let $\sigma = \{\ra{x_1},\ra{x_2}, \ldots, \ra{x_n}\}$, with $\ra{r} \leq \ra{x_1}$. Then $\ra{x_1} \vee \ra{s} \in \sigma' \subseteq O$. Since $O$ is strongly consistent, this implies that $\ra{x_1}$ and $\ra{s}$ lie in $O$. Also, for any $i \geq 2$, as $x_i \in S$, either $\ra{x_i}$ or $\la{x_i}$ lies in $O$. However, since $\ra{s} \in O$ and $\ra{x_i} \wedge \la{s} \in \sigma' \subseteq O$, and $O$ avoids $\cc{P}$, it cannot be the case that $\la{x_i} \in O$. Hence $\ra{x_i} \in O$ for all $i \geq 2$, and so $\sigma \subseteq O$.
\end{proof}

\begin{lemma}\label{l:close}
Let  $\cc{F}^*$ be an uncrossing of~$\cc{F}$. Then an orientation $O$ of $S$ is a regular $\hat{\cc{F}^*}$-tangle if and only if it is a regular $\cc{F}$-tangle.
\end{lemma}
\begin{proof}
Recall that $O$ is a regular $\cc{F}$-tangle if and only if it is a regular $\cc{F}^*$ tangle (Lemma \ref{l:*tangle}). Clearly every regular $\hat{\cc{F}^*}$-tangle also avoids  $\cc{F}^*$, and hence is also a regular $\cc{F}^*$-tangle, and $\cc{F}$-tangle.

Conversely, every regular $\cc{F}$-tangle avoids both $\cc{F}^* = \cc{G}_0$ (Lemma~\ref{l:*tangle}) and hence, by Lemma \ref{l:shift} and $\cc{F} \supseteq \cc{P}_S$, also $\cc{G}_1$. Proceeding inductively we see that $O$ avoids $\cc{G}_n$ for each $n$, and so avoids $\bigcup_n \cc{G}_n = \hat{\cc{F}^*}$. Hence $O$ is a regular $\hat{\cc{F}^*}$-tangle.
\end{proof}

Before we can, at last, apply Theorem \ref{t:dual} to our set $\hat{\cc{F}^*}$, we have to make one final adjustment: Theorem \ref{t:dual} requires its set $\cc{F}$ of stars to be standard, i.e., to contain all singletons stars $\{\ra{r}\}$ with $\la{r}$ trivial in $\ra{S}$. Since trivial separations are small, it will suffice to add to $\hat{\cc{F}^*}$ all singleton stars $\{\ra{x}\}$ such that $\la{x}$ is small; we denote the resulting superset of $\hat{\cc{F}^*}$ by $\overline{\cc{F}^*}$. Then $\overline{\F^*}$-tangles contain all small separations, so they are precisely the regular $\hat\F^*$-tangles.%
   \COMMENT{}

Clearly, $\overline{\cc{F}^*}$ is a standard set of stars. Also, the shift of any singleton star $\{\ra{x}\}$ is again a singleton star $\{\ra{y}\}$ such that $\ra{x} \leq \ra{y}$. Moreover, if $\la{x}$ is small then so is $\la{y} \leq \la{x}$, so if $\{\ra{x}\}$ lies in $\overline{\cc{F}^*}$ then so does $\{\ra{y}\}$. Therefore, since $\hat{\cc{F}^*}$ is closed under shifting, $\overline{\cc{F}^*}$ too is closed under shifting.  Hence, we get the following duality theorem for abstract profiles:

\begin{theorem}\label{t:main}
Let $\ra{S}$ be a separable%
   \footnote{Whilst the assumption that $\ra{S}$ is separable is necessary to apply Theorem \ref{t:dual}, in a forthcoming paper \cite{AbstractTangles} the authors and Wei{\ss}auer show that every submodular separation system is in fact separable, and so this asssumption can be removed from Theorem \ref{t:main}.}
submodular separation system in some universe of separations, let $\cc{F}\subseteq 2^{\ra{S}}$ contain~$\cc{P}_S$, and let $\cc{F}^*$ be any uncrossing of~$\cc{F}$. Then the following are equivalent:
\goodbreak
\begin{itemize}
\item There is no regular $\cc{F}$-tangle of $S$;
\item There is no $\overline{\cc{F}^*}$-tangle of $S$;
\item There is an $S$-tree over $\overline{\cc{F}^*}$.
\end{itemize}
\end{theorem}
\begin{proof}
By Lemmas \ref{l:*tangle} to \ref{l:close} the regular $\cc{F}$-tangles of $S$ are precisely its regular $\hat{\cc{F}^*}$-tangles, and by Lemma \ref{l:small} these are precisely its $\overline{\cc{F}^*}$-tangles. Hence the first two statements are equivalent. Since $\overline{\cc{F}^*}$ is a standard set of stars which is closed under shifting, Theorem \ref{t:dual} implies that the second two statements are equivalent.
\end{proof}

%Theorem \ref{t:main} implies duality theorems for a large class of regular $\F$-tangles in graphs:%
%  \COMMENT{} 
%  all those that are $k$-profiles because $\F$ includes~$\P_k$, explicitly or implicitly.%
%   \COMMENT{}
%   For example, we might wish for a duality theorem for some interesting subclass of the classical $k$-tangles. If this subclass can be defined by extending the set $\T$ that specifies classical tangles~\eqref{tangleaxiom} to some larger set~$\F$, then this $\F$ will still satisfy the premise in Theorem~\ref{t:main}, which would thus yield a duality theorem for this particular type of $k$-tangle.

\section{Duality for special tangles, blocks, and profile\rlap s}\label{s:apply}

Let us now apply Theorem~\ref{t:main} to prove Theorems~\ref{t:Ftangle}--\ref{t:prof} from the Introduction. 
We shall first state the latter two results by specifying~$\F$, and then deduce their \td\ formulations as in Theorems~\ref{t:block} and~\ref{t:prof}, along with Theorem~\ref{t:Ftangle}. Recall that for any graph the set $\ra{S_k}$ of separations of order $<k$ is a separable submodular separation system (see \cite{TangleTreeGraphsMatroids}).

\begin{theorem}\label{t:blocks}
For every finite graph $G$ and every integer $k > 0$ exactly one of the following statements holds:
\begin{itemize}\itemsep=0pt
\item $G$ contains a $k$-block;
\item $G$ has an $S_k$-tree over~$\overline{\B_k^*}$, where $\B_k^*$ is any uncrossing of~$\B_k$.
\end{itemize}
\end{theorem}

\begin{proof}
In Theorem~\ref{t:main}, let $S = S_k$ be the set of separations of order~$<k$ in~$G$, and let $\cc{F} := \cc{B}_k \cap 2^{\ra{S_k}}$. Then $\ra{S_k}$ is a separable submodular separation system in the universe of all separations of~$G$, and the regular $\F$-tangles in $G$ are precisely the orientations~$O(b)$ of~$S_k$ for $k$-blocks $b$ in~$G$. Hence $G$ has a $k$-block if and only if it has a regular $\F$-tangle for this~$\F$. The assertion now follows from Theorem~\ref{t:main}.
\end{proof}

As for profiles, every graph $G$ has a regular 1-profile~-- just orient every 0-separation towards some fixed component~--%
   \COMMENT{}
   and a regular 2-profile (Lemma~\ref{l:two}). So we need a duality theorem only for $k > 2$. Recall that, for $k>2$, all $k$-profiles of graphs are regular (Lemma~\ref{l:strong}).

\begin{theorem}\label{t:profs}
For every finite graph $G$ and every integer $k > 2$ exactly one of the following statements holds:
\begin{itemize}\itemsep=0pt\vskip-\smallskipamount\vskip0pt
\item $G$ has a $k$-profile;
\item $G$ has an $S_k$-tree over~$\overline{\P_k^*}$, where $\P^*_k$ is any uncrossing of~$\P_k$.
\end{itemize}
\end{theorem}

\begin{proof}
In Theorem~\ref{t:main}, let $S = S_k$ be the set of separations of order~$<k$ in~$G$, and let $\cc{F} := \cc{P}_k $. Then $\ra{S_k}$ is a separable submodular separation system in the universe of all separations of~$G$, and the regular $\F$-tangles in $G$ are precisely its $k$-profiles.%
   \COMMENT{}
   The assertion now follows from Theorem~\ref{t:main}, as before.
\end{proof}

Theorems \ref{t:Ftangle}--\ref{t:prof} now follow easily: we just have to translate $S$-trees into \td s.

\medbreak

\noindent {\bf Proof of Theorems \ref{t:Ftangle}--\ref{t:prof}.}
Given a set $S$ of separations of~$G$ and an $S$-tree $(T,\alpha)$, with $\alpha(\ra{e}) =: (A_\alpha (\ra{e}), B_\alpha (\ra{e}))$ say, we obtain a \td\ $(T,\V_\alpha)$ of~$G$ with $\V_\alpha = (V_t)_{t\in T}$ by letting
 $$V_t := \bigcap\, \{\, B_\alpha (\ra{e})\mid \ra{e} = (s,t)\in\ra{E(T)}\,\}.$$
 Note that $(T,\alpha)$ can be recovered%
   \COMMENT{}
   from this \td: given just $T$ and~$\V = (V_t)_{t\in T}$, we let $\alpha$ map each oriented edge $\ra{e} = (t_1,t_2)$ of~$T$ to the oriented separation of~$G$ it induces: the separation
 $\big(\bigcup_{t \in T_1} V_t \,,\, \bigcup_{t \in T_2} V_t\big)$
  where $T_i$ is the component of $T-e$ containing~$t_i$.

Recall that the set $\T$ defined in~\eqref{tangleaxiom} contains~$\P$. Hence so does any $\F\supseteq\T$. For such~$\F$, therefore, every $\F$-tangle of~$S_k$ is a $k$-profile, and hence is regular by Lemma~\ref{l:strong} if $k>2$. For $\F_k := \F\cap 2^{\ra{S_k}}$ and
 $$\T_\F (k) := \{\,(T,\V_\alpha)\mid (T,\alpha)\text{ is an $S_k$-tree over } \overline{\F_k^*}\,\}$$
we thus obtain Theorem~\ref{t:Ftangle} directly from Theorem~\ref{t:main}. Similarly, letting
 \begin{eqnarray*}
 \T_\B (k) &:=& \{\,(T,\V_\alpha)\mid (T,\alpha)\text{ is an $S_k$-tree over } \overline{\B_k^*}\,\}\\
 \T_\P (k) &:=& \{\,(T,\V_\alpha)\mid (T,\alpha)\text{ is an $S_k$-tree over } \overline{\P_k^*}\,\}
 \end{eqnarray*}
 yields Theorems \ref{t:block} and~\ref{t:prof} as corollaries of Theorems \ref{t:blocks} and~\ref{t:profs}.\qed

\section{Width parameters}\label{sec:width}

In this section we derive some bounds for the block and profile width of a graph that follow easily from our main results combined with those of~\cite{TangleTreeAbstract, TangleTreeGraphsMatroids} and~\cite{GMX}.

Given a star $\sigma$ of separations in a graph, let us call the set $\bigcap \{\,B\mid (A,B)\in\sigma\}$ the {\em interior\/} of~$\sigma$ in~$G$. For example, every star in $\cc{P}_k^*$ is of the form
\[
\{ (A,B), (B \cap C, A \cup D), (B \cap D, A \cup C) \} \subseteq S_k,
\]
and hence every vertex of its interior lies in at least two of the separators $A \cap B$, $(B \cap C) \cap (A \cup D)$ and $(B \cap D) \cap (A \cup C)$. Since all these separations are in~$S_k$, the interior of any star in $\cc{P}_k^*$ thus has size at most $3(k-1) /2$.

We can apply this observation to obtain the following upper bound on the profile-width $\pw(G)$ of a graph~$G$ in terms of its tree-width~$\tw(G)$:

\begin{theorem}\label{t:profwidth}
For every graph $G$,
\begin{equation}\label{twbound}\textstyle
\pw(G) \leq \tw{}(G) + 1 \leq \frac{3}{2}\pw{}(G) .
\end{equation}
\end{theorem}

\begin{proof}
For the first inequality, note that $\tw(G)+1$ is the largest integer~$k$ such that $G$ has no \td\ into parts of order~$<k$. By the duality theorem for tree-width from~\cite{TangleTreeGraphsMatroids}, having no such \td\ is equivalent to admitting an $\S^k$-tangle of~$S_k$, where
 $$\cc{S}^n = \big\{\, \tau\subseteq\ra{U} : \tau\text{ is a star and } \big| \bigcap\{\,B: (A,B)\in\tau\,\} \big| < n\,\big\}$$
and $\ra{U}$ is the universe of all separations of~$G$. But among the $\S^k$-tangles of~$S_k$ are all the $k$-profiles of~$G$. (An easy induction on~$|\tau|$ shows that a regular%
   \COMMENT{}
   $k$-profile has no subset $\tau\in\S^k$; cf.\ Lemma~\ref{l:strong} and \cite[Prop.\,3.4]{ProfilesNew}.)%
  \COMMENT{} 
   Therefore $G$ has no $k$-profile for $k > \tw{}(G)+1$, which Corollary~\ref{invariants} translates into $\pw(G) \leq \tw{}(G) +1$.%
   \COMMENT{}

For the second inequality, recall that if $k := \pw(G)$ then $G$ has a \td\ in~$\T_\P(k+1)$. The parts of this \td\ are interiors of stars in~$\overline{\P_{k+1}^*}$, so they have size at most~$3k/2$.%
   \COMMENT{}
   This \td, therefore, has width at most $(3k/2)-1$, which thus is an upper bound for~$\tw(G)$.
\end{proof}

We can also relate the profile-width of a graph to its branch-width, as follows. In order to avoid tedious exceptions for small~$k$, let us define the {\em adjusted branch-width\/} of a graph~$G$ as
 $$\brw(G) := \min \{\,k\mid G\text{ has no } S_{k+1}\text{-tree over }\T^*\}. $$
 By~\cite[Theorem 4.4]{TangleTreeGraphsMatroids}, this is equivalent to the {\em tangle number\/} of~$G$, the greatest~$k$ such that $G$ has a $k$-tangle. For $k\ge 3$ it coincides with the original branch-width of~$G$ as defined by Robertson and Seymour~\cite{GMX}.%
   \footnote{Our {\em adjusted branch-width\/} is the dual parameter to the tangle number for all~$k$, while the original branch-width from~\cite{GMX} achieves this only for $k\ge 3$: it deviates from the tangle number for some graphs and $k\le 2$. See~\cite[end of Section~4]{TangleTreeGraphsMatroids} for a discussion.}

Robertson and Seymour \cite{GMX} showed that the adjusted branch-width of a graph is related to its tree-width in the same way as we found that the profile-width is:
\begin{equation}\label{brwbound}\textstyle
\brw{}(G)  \leq \tw{}(G) + 1 \leq \frac{3}{2}\brw{}(G) .
\end{equation}

Together, these inequalities imply the following relationship between branch-width and profile-width:

\begin{corollary}\label{c:branchwidth}
For every graph $G$,
\begin{equation}\textstyle
\brw{}(G) \leq \pw{}(G) \leq \frac{3}{2}\brw{}(G).
\end{equation}
\end{corollary}

\begin{proof}
The first inequality follows from the fact that $k$-tangles are $k$-profiles, and that the largest $k$ for which $G$ has a $k$-tangle equals the adjusted branch-width: thus,
$$\brw(G) = k\le\pi(G) = \pw(G)$$
by Corollary~\ref{invariants}.

For the second inequality, notice that $\pw{}(G) \specrel\leq{\eqref{twbound}} \tw{}(G) + 1 \specrel\leq{\eqref{brwbound}} \frac{3}{2}\brw{}(G)$.
\end{proof}

Since every $k$-block defines a $k$-profile, Corollary~\ref{invariants} implies%
   \COMMENT{}
   that the block-width of a graph is a lower bound for its profile-width, and hence by~\eqref{twbound} also for it tree-width (plus~1). Conversely, however, no function of the tree-width of a graph can be a lower bound for its block-width. Indeed, the tree-width of the $k \times k$-grid $H_k$ is at least~$k$ (see \cite{DiestelBook16}) but $H_k$~contains no $5$-block: in every set of $\geq 5$ vertices there are two non-adjacent vertices, and the neighbourhood of either vertex is then a set of size~$4$ which separates the two vertices. Therefore there exist graphs with bounded block-width and arbitrarily high tree-width.

Since large enough tree-width forces both large profile-width~\eqref{twbound} and large branch-width~\eqref{brwbound}, the grid example shows further that making these latter parameters large cannot force the block-width of a graph above~4.

\bibliographystyle{plain}
\bibliography{collective}

\begin{thebibliography}{10}

\bibitem{ForcingBlocks}
J.~Carmesin, R.~Diestel, M.~Hamann, and F.~Hundertmark.
\newblock $k$-{B}locks: a connectivity invariant for graphs.
\newblock {\em SIAM J.\ Discrete Math.}, 28:1876--1891, 2014.

\bibitem{confing}
J.~Carmesin, R.~Diestel, F.~Hundertmark, and M.~Stein.
\newblock Connectivity and tree structure in finite graphs.
\newblock {\em Combinatorica}, 34(1):1--35, 2014.

\bibitem{CG14:isolatingblocks}
J.~Carmesin and P.~Gollin.
\newblock Canonical tree-decompositions of a graph that display its $k$-blocks.
\newblock {\em J. Combin. Theory Ser. B}, 122:1--20, 2017.

\bibitem{AbstractSepSys}
R.~Diestel.
\newblock Abstract separation systems.
\newblock To appear in {\em Order} (2017), DOI 10.1007/s11083-017-9424-5,
  arXiv:1406.3797v5.

\bibitem{DiestelBook16}
R.~Diestel.
\newblock {\em Graph Theory \emph{(5th edition, 2016)}}.
\newblock Springer-Verlag, 2017.
\newblock \\ Electronic edition available at {\small\tt
  http://diestel-graph-theory.com/}.

\bibitem{AbstractTangles}
R.~Diestel, J.~Erde, and D.~Wei{\ss}auer.
\newblock Tangles in abstract separation systems.
\newblock In preparation.

\bibitem{ProfilesNew}
R.~Diestel, F.~Hundertmark, and S.~Lemanczyk.
\newblock Profiles of separations: in graphs, matroids, and beyond.
\newblock arXiv:1110.6207, to appear in {\em Combinatorica}.

\bibitem{DiestelOumDualityII}
R.~Diestel and S.~Oum.
\newblock Unifying duality theorems for width parameters, {II.~G}eneral
  duality.
\newblock arXiv:1406.3798, 2014.

\bibitem{TangleTreeAbstract}
R.~Diestel and S.~Oum.
\newblock Tangle-tree duality in abstract separation systems.
\newblock arXiv:1701.02509, 2017.

\bibitem{TangleTreeGraphsMatroids}
R.~Diestel and S.~Oum.
\newblock Tangle-tree duality in graphs, matroids and beyond.
\newblock arXiv:1701.02651, 2017.

\bibitem{MonaLisa}
R.~Diestel and G.~Whittle.
\newblock Tangles and the {M}ona {L}isa.
\newblock arXiv:1603.06652.

\bibitem{mader78}
W.~Mader.
\newblock {\"U}ber $n$-fach zusammenh\"angende {E}ckenmengen in {G}raphen.
\newblock {\em J.~Combin.\ Theory (Series B)}, 25:74--93, 1978.

\bibitem{GMX}
N.~Robertson and P.D. Seymour.
\newblock Graph minors. {X}. {O}bstructions to tree-decomposition.
\newblock {\em J.~Combin.\ Theory (Series B)}, 52:153--190, 1991.

\end{thebibliography}

\end{document}